\documentclass[11pt]{elsarticle} 
\biboptions{authoryear}
\usepackage{setspace}
\journal{European Journal of Operational Research}

\usepackage[utf8]{inputenc} 

\usepackage{subfigure}
\usepackage{chemformula}
\usepackage[margin=1in]{geometry}
\usepackage{graphicx} 
\usepackage{epstopdf}

\usepackage{booktabs} 
\usepackage{array} 
\usepackage{paralist} 
\usepackage{verbatim} 
\usepackage{subfig} 

\usepackage{fancyhdr} 
\usepackage[nottoc,notlof,notlot]{tocbibind} 
\usepackage[titles,subfigure]{tocloft} 

\usepackage{amsmath, amsthm, amssymb}

\usepackage{multirow}
\usepackage{algorithm}
\usepackage{algorithmic}

\usepackage{amsthm}

\newtheorem{proposition}{Proposition}

\newtheorem{notation}{Notation}

\allowdisplaybreaks

\DeclareMathOperator*{\conv}{conv}

\usepackage{rotating}
\usepackage{pdflscape}
\usepackage{pgfplots}
\usepgfplotslibrary{units}
\usepackage{url}
\usepackage{color}

\graphicspath{{figures/}} 

\allowdisplaybreaks

\usepackage{color}
\usepackage{bm}
\usepackage{bbm}

\newcommand{\ii}{\mathrm{i}}

\newcommand{\cA}{\mathcal{A}}

\newcommand{\cS}{\mathcal{S}}

\newcommand{\cY}{\mathcal{Y}}

\newcommand{\cN}{\mathcal{N}}
\newcommand{\cM}{\mathcal{M}}
\newcommand{\cT}{\mathcal{T}}

\newcommand{\cU}{\mathcal{U}}

\begin{document}

\begin{frontmatter}
\title{Optimization Problems Involving Matrix Multiplication \\ with Applications in Material Science and Biology
}

\author{Burak Kocuk} 
\address{Industrial Engineering Program,  Sabanc{\i} University, Istanbul, Turkey 34956\\burak.kocuk@sabanciuniv.edu}

\begin{abstract}
We consider optimization problems involving the multiplication of variable matrices to be selected from a given family, which might be a discrete set, a continuous set or a combination of both. Such nonlinear, and possibly discrete,  optimization problems arise in applications from biology and material science among others, and are known to be NP-Hard  for a special case of interest. We analyze the underlying structure of such optimization problems for two particular applications and, depending on the matrix family, obtain compact-size mixed-integer linear or quadratically constrained quadratic programming reformulations that can be solved via commercial solvers. Finally, we present the results of our computational experiments, which demonstrate the success of our approach compared to heuristic and enumeration methods predominant in the literature.
\end{abstract}

\end{frontmatter}

%

\section{Introduction}
\label{sec:intro}

Consider an optimization problem of the following form:
\begin{subequations} \label{eq:generic}
\begin{align}
\max_{T, w} \ & f( w ) \\
\text{s.t.} \ & p  T_1 \cdots  T_N = w  \label{eq:generic cons} \\
\ & T_1,\dots,T_N \in \cT. \label{eq:generic domain T} 
\end{align}
\end{subequations}
Here,  $p \in\mathbb{C}^{r \times d}$ is a given matrix, $f: \mathbb{C}^{r \times d} \to \mathbb{R}$ is a function  and $\cT \subseteq \mathbb{C}^{d \times d}$ is a family of matrices, which may be a discrete set, a continuous set or a combination of both. Observe that~\eqref{eq:generic} is a  nonlinear optimization problem since constraint~\eqref{eq:generic cons}  contains the multiplication of $N$ variable matrices, resulting in  degree-$N$ polynomials. Depending on the structure of the set~$\cT$, this optimization problem may also involve discrete decisions.
 Such optimization problems naturally arise in material science and biology, and  a special case in which  $f$ is a linear function and $\cT$ is a finite set  is known to be NP-Hard   \citep{tran2017antibiotics}.

The above abstract problem setting can be interpreted as follows: Suppose that there is a ``system'' initialized with the given matrix $p$. Then, the decision maker chooses a matrix $T_1$ and the system ``evolves'' to another state $p T_1$.  This process continues for $N$ transitions. Finally, the ``performance'' of the decisions $T_1,\dots,T_N$ is computed via the function $f$, whose argument is the final system state $w$.
We will now give concrete examples from material science and biology, and motivate the significance of analyzing optimization problems involving matrix multiplication. 

The first example is from material science and is called \textit{the multi-layer thin films problem}. Reflectance is an important electromagnetic property of materials and, in many optics applications, materials with high reflectance are desired. When  reflectance  of a metallic substrate is not satisfactory, dielectric coating materials can be used for   enhancement. For instance,   reflectance   of Tungsten at  450 nanometers (nm) wavelength is approximately 47\% but it can be increased to 87\% if one layer each of Titanium Dioxide and Magnesium Fluoride thin films with a quarter wavelength optical thicknesses are coated on top. Given a material library and  a budget on the number of layers, the multi-layer thin films problem seeks to find the optimal configuration of dielectric coating materials and their thicknesses to be coated in each layer so that the reflectance is maximized.  This classical problem in optics is typically solved via  heuristic and metaheuristic methods \citep{macleod2010thin, pedrotti2017introduction, tikhonravov1996application, hobson2004markov, rabady2014global, shi2017optimization, keccebacs2018enhancing}, and the rigorous treatment of the underlying optimization problem is lacking in the literature.

The second example arises from biology and is called \textit{the antibiotics time machine problem}. Antibiotic drug resistance is a serious concern in modern medical practices since the successive application of antibiotics may cause mutations, which might lead to ineffective or even harmful treatment plans. Even further complicating the problem is the inherent randomness associated with administering a certain drug.
Given a list of drugs and a predetermined length of the treatment plan, the antibiotics time machine problem seeks to find the optimal drug sequence to be applied so that the probability of reversing the mutations altogether at the end of the treatment is maximized. 
Although there is significant interest in the biology community to understand the quantitative aspect of antibiotics resistance \citep{bergstrom2004ecological, kim2014alternating, nichol2015steering, mira2015rational, mira2017statistical, yoshida2017time}, the only  method used to attack the antibiotic times machine problem appears to be complete enumeration. 

These two seemingly unrelated optimization problems can, in fact, be formulated as in \eqref{eq:generic}. In the case of   the multi-layer thin films problem, the matrix collection $\cT$ is a mixed-integer nonlinear set and the objective function $f$ is the ratio of two convex quadratic functions whereas, in the antibiotics time machine problem, $\cT$ is a finite set and $f$ is a linear function. One of our main contributions in this paper is that the generic nonlinear discrete optimization problem \eqref{eq:generic} can be reformulated as a  mixed-integer quadratically constrained quadratic program (MIQCQP) for the former problem, and a mixed-integer linear program (MILP) for the latter problem. These reformulations allow us to solve the practically relevant instances of both problems to global optimality using commercial optimization packages.

As mentioned above, literature mostly focuses on heuristics methods and complete enumeration to solve optimization problems involving matrix multiplication and the rigorous analysis of these problems from an applied operations research perspective is insufficient. One exception in this direction is \citet{wu2018optimal}, in which  the collection $\cT$ is assumed to be  a finite set. The authors provide sufficient conditions for the polynomial-time solvable cases of problem \eqref{eq:generic}, which are quite restrictive from an application point of view. In contrast, our approach in this paper is application-driven and computational, and focuses on developing methods to solve practical instances of  problem \eqref{eq:generic} arising from real-life applications.
Moreover, it is also possible to utilize our approach to attack other applications with similar structure as reported in   \citet{wu2018optimal},
  including the matrix mortality problem \citep{blondel1997pair, bournez2002mortality} and the joint spectral radius computation \citep{rota1960note,jungers2009joint}.
 

%
%
%
%


The rest of the paper is organized as follows: {In Section~\ref{sec:main}, we provide reformulations of the feasible region of problem~\eqref{eq:generic}
 depending on the structure of the set $\cT$.}
Then, we specialize these general results to two applications, multi-layer thin films from material science in Section~\ref{sec:thinFilms} and antibiotics time machine from biology in Section~\ref{sec:antibiotic}, and present the results of our extensive computational experiments.
Finally, we conclude our paper in
Section~\ref{sec:conc} with final remarks and future research directions.

\section{General Results}
\label{sec:main}

In this section, we analyze problem~\eqref{eq:generic} and propose its reformulations based on the structure of the set~$\cT$. 
In particular, we first  provide a straightforward bilinear reformulation of the polynomial constraint~\eqref{eq:generic cons} in 
Section~\ref{sec:mainBilinear}. Under the assumption that $\cT$ is a finite set, we further reformulate the feasible region of  problem~\eqref{eq:generic} as a mixed-integer linear representable set  in Section~\ref{sec:mainLinear}.



\subsection{Bilinearization}
\label{sec:mainBilinear}

Let us  define a set of matrix variables $u_n \in \mathbb{C}^{r \times d}$, $n=0,\dots,N$  that satisfy  the recursion  $u_n = u_{n-1} T_n$ for $n=1,\dots,N$ with $u_0:=p$. Then, problem \eqref{eq:generic} can be reformulated as follows:
\begin{subequations} \label{eq:genericB}
\begin{align}
\max_{T, w, u} \  &   f ( w )  \\
\textup{s.t.} \  \ &  u_{n-1} T_n =  u_{n}  \qquad  n = 1, \dots,  N \label{eq:genericB bilinear}  \\
\ & u_0 = p , u_N = w \label{eq:genericB boundary} \\
\ &  \eqref{eq:generic domain T}  \notag. 
\end{align}
\end{subequations}


We remark that provided that the matrix family $\cT$ is a bounded set,  each $u_n$ matrix is guaranteed to come from another bounded set $\cU_n\subseteq \mathbb{C}^{ r\times d}$  defined as
\begin{equation}\label{eq:define Un}
\cU_n :=  \bigcup_{k=1}^K\{ u_{n-1} \hat T_k : u_{n-1}\in \cU_{n-1}\},
\end{equation}
for $n=1,\dots,n$,  with  $\cU_0:=\{p\}$. 
This observation is quite important from the following aspect: The boundedness of the set $\cU_n$ can be utilized to construct relaxations for the bilinear constraint~\eqref{eq:genericB bilinear} in a straightforward manner (e.g. one can obtain a McCormick-based relaxation \citep{mccormick1976computability}  once variable bounds for each entry of the unknown matrices are available). Moreover, under the assumption that the matrix family $\cT$ is a finite set and a polyhedral outer-approximation of $\cU_n$ is utilized, then an \textit{equivalent} mixed-integer linear representation of the feasible region of 
problem~\eqref{eq:genericB} can be obtained, as discussed in the next section.

\subsection{Linearization}
\label{sec:mainLinear}

In this part, we will assume that the set $\cT$ is a finite set given as $\cT:=\{\hat T_k : k=1,\dots,K\}$. Under this assumption, the resulting bilinear discrete optimization problem~\eqref{eq:genericB} obtained at the end of the bilinearization step can be further reformulated such that its feasible region is mixed-integer linear representable. We now introduce Proposition \ref{prop:linearization}, which will be crucial in the sequel.

\begin{proposition}\label{prop:linearization}
Given a finite collection of matrices  $\cA =\{\hat A_k : k=1,\dots, K\} \subseteq \mathbb{C}^{d \times d}$ and a polytope $\cY \subseteq \mathbb{C}^{r \times d}$, consider the set 
\[
\cS := \left\{(y, A, z)\in \cY \times \cA \times  \mathbb{C}^{r \times d} :  y A = z \right\}.
\]
Then, the following statements hold:
\begin{enumerate}
\item
The  system \eqref{eq:ext} in variables $(y, A, z, v_k, \mu_k)$ is an extended formulation for $\cS$:
\begin{subequations}\label{eq:ext}
\begin{align}
 \sum_{k=1}^K \mu_k  \hat A_k &= A  \label{eq:ext1} \\ 
\sum_{k=1}^K v_{k} &= y  \label{eq:ext2}  \\ 
 \sum_{k=1}^K v_{k} \hat A_k  &=z  \label{eq:ext3} \\ 
\sum_{k = 1}^K \mu_{k} &=  1  \label{eq:ext4}  \\
 v_{k} \in \cY \mu_k, \ \mu_k &\in\{0,1\}, \  k=1,\dots,K.
\end{align}
\end{subequations}
\item We have
\[
\conv(\cS) =  \left\{ (y, A, z)\in \cY \times \cA \times  \mathbb{C}^{r \times d} : \ \exists ( v_{k},\mu_k) \in \cY \mu_k \times \mathbb{R}_+    : \eqref{eq:ext1}-\eqref{eq:ext4}  \right\}.
\]
\end{enumerate}
\end{proposition}
\begin{proof}
Let us define the sets
\[
\cS_k := \left\{(y, A, z)\in \cY \times \cA \times \mathbb{C}^{r \times d} :  y \hat A_k = z, A=\hat A_k \right\},
\]
for each $k=1,\dots,K$. Clearly, we have that $\cS = \cup_{k=1}^K \cS_k$. The statements of the proposition follow by constructing a $K$-way disjunction of $\cS$ due to \citet{balas1979disjunctive}.
\end{proof}

We now apply Proposition \ref{prop:linearization} to problem  \eqref{eq:genericB} by setting $\cA=\cT$, $\hat A_k=\hat T_k$, $\cY=\bar \cU_{n-1}$, $y=u_{n-1}$, $A=T_n$ and $z=u_n$ for $n=1,\dots,N$.
After defining the copy variables $v_{n,k} \in \mathbb{C}^{r\times d}$ and binary variables $x_{n-1,k}$ which take value one if $T_n =\hat T_k$ and zero otherwise, we obtain the following problem with a mixed-integer linear representable feasible region:
\begin{subequations}\label{eq:MILP}
\begin{align}
\max_{u,v,x} \  &   f(w)   \label{eq:objMILP}\\
\textup{s.t.} \  & \sum_{k=1}^K v_{n-1,k} = u_{n-1} &n&=1,\dots,N  \label{eq:MILP1}  \\ 
 \  &  \sum_{k=1}^K v_{n-1,k} \hat T_k  =u_n   &n&=1,\dots,N   \label{eq:MILP2} \\ 
 \  & \sum_{k = 1}^K x_{n-1,k} =  1   &n&=1,\dots,N   \label{eq:MILP3}  \\
 \  &  \ v_{n-1,k} \in  \bar \cU_{n-1} x_{n-1,k}, \ x_{n-1,k} \in\{0,1\}    &n&=1,\dots,N,  k=1,\dots,K \\
 \ &  \eqref{eq:genericB boundary}. \notag
\end{align}
\end{subequations}
In this formulation, the relation $v_{n-1,k} \in \bar \cU_{n-1} x_{n-1,k}$ serves as a ``big-$M$ constraint''. Here, any polyhedral set $\bar \cU_n$ that outer-approximates the set $\cU_n$ can  be used without changing the feasible region of problem \eqref{eq:MILP}. {When applicable, we  provide such reasonable sets in the formulations of the specific applications considered in the remainder of this paper.}

\section{An Application from Material Science: Multi-Layer Thin Films}
\label{sec:thinFilms}

In this section, we  study the multi-layer thin films  problem from material science. We first introduce the basic notions in optics and provide a formal problem definition in Section~\ref{sec:thinFilmsDef}. Then,  we propose an MIQCQP formulation in Section~\ref{sec:thinFilmsForm} 
and its enhancements in Section~\ref{sec:thinFilmsFormEnhance}. 
We overview a commonly used heuristic from literature and discuss its convergence behavior  in Section~\ref{sec:thinFilmsHeur}. 
Finally, we present the results of our computational experiments in Section \ref{sec:thinFilmsComp}, which include a  discussion on the effect of formulation enhancements and a comparison of the optimal solutions with the heuristic ones.

\subsection{Problem Definition}
\label{sec:thinFilmsDef}

Suppose that we have a metallic substrate and our aim is to increase its reflectance  by coating a set of dielectric materials on top. 
Following the classical textbooks on optics~\citep{macleod2010thin, pedrotti2017introduction}, we will first introduce the basic concepts and notations in multi-layer thin films, and then present how we can attack this problem using optimization techniques.

Let us denote the  refractive   index of a metallic substrate (e.g. Tungsten, Tantalum, Molybdenum, Niobium) at wavelength $\lambda$ as $\hat a_s^\lambda  \in \mathbb{C}$, where the imaginary part is a measure of reflection losses.
Let $\cM$ be the set of dielectric coating materials, such as Silicon Dioxide (\ch{SiO2}), Titanium Dioxide  (\ch{TiO2}),  Magnesium Fluoride  (\ch{MgF2}), Aluminum Oxide  (\ch{Al2O3}). For a given wavelength~$\lambda$, we will denote the set of refractive  indices\footnote{In reality, the refractive  index of  a dielectric coating material is also a complex number. However, since the reflection loss of a dielectric material is negligibly small, we will ignore the imaginary part of this complex number.} by
\[
\cA^\lambda  := \{\hat  a_m^\lambda : m \in \cM\} \subseteq \mathbb{R}_+. 
\]

We will now introduce an important concept called the \textit{transfer matrix}, which is used to quantify the reflectance through a material.
Under the assumption that the light is at normal incidence, the transfer matrix of material $m$ of thickness $t$ at wavelength $\lambda$  is given as
\begin{equation}\label{eq:transferM}
T_{m,t}^\lambda =
 \begin{bmatrix}
\cos  \sigma_{m,t}^\lambda   &  \ii \frac{\sin \sigma_{m,t}^\lambda}{ \hat  a_m^\lambda} \\
 \ii {\hat a_m^\lambda}{\sin \sigma_{m,t}^\lambda} & \cos  \sigma_{m,t}^\lambda
\end{bmatrix}, \quad \text{ where } \sigma_{m,t}^\lambda = \frac{2\pi \hat a_m^\lambda t}{\lambda} .
\end{equation}
Here, $\ii=\sqrt{-1}$. An important fact related to transfer matrices is their ``multiplicative'' property, that is, the cumulative effect a coating material $m_1$ of thickness $t_1$ on top of a coating material $m_2$ of thickness $t_2$ is simply obtained by the product of their own transfer matrices 
$T_{m_1,t_1}^\lambda  T_{m_2,t_2}^\lambda  $ (see Figure \ref{fig:thinFilms} for an illustration). 


%
%
%
%
%
%
%
%

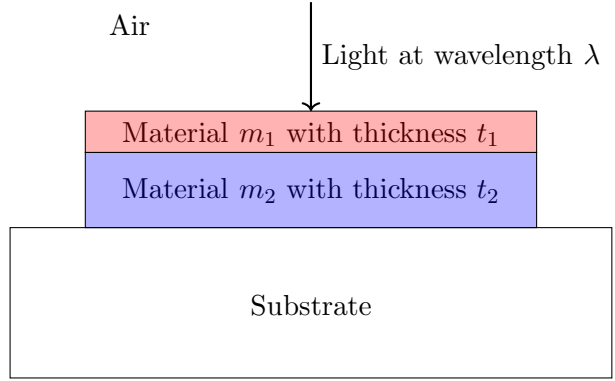
\begin{figure}[H]
\centering

\begin{tikzpicture}

\coordinate (Air) at (2,4.7);
\draw (Air)  node[left] {Air\qquad};

\coordinate (Airr) at (4,5);

\coordinate (Mat0) at (4,3.55);

\coordinate (Substrate) at (4,0.7);
\draw (Substrate)  node[above] {Substrate};

\coordinate (Mat1) at (4,2.95);
\draw (Mat1)  node[above] {Material $m_1$ with thickness $t_1$}; 

\coordinate (Mat2) at (4,2.2);
\draw (Mat2)  node[above] {Material $m_2$ with thickness $t_2$};

\draw [draw=black] (0,0) rectangle (8,2);
\filldraw [fill=blue, fill opacity=0.3, draw=black] (1,2) rectangle (7,3);
\filldraw [fill=red, fill opacity=0.3, draw=black] (1,3) rectangle (7,3.55);

\path [->] (Airr) edge[thick] node[align=center,right]{Light at wavelength $\lambda$} (Mat0);
\end{tikzpicture}
\caption{Illustration of a multi-layer thin film with $N=2$ layers.}\label{fig:thinFilms}
\end{figure}

%

We will denote the set of all transfer matrices obtainable from coating materials in $\cM$ at wavelength $\lambda$  as
\[
\cT^\lambda_+ :=
\left\{  \begin{bmatrix}
\cos  \sigma    &  \ii \frac{\sin \sigma}{a} \\
 \ii {a}{\sin \sigma} & \cos  \sigma
\end{bmatrix}: \sigma = \frac{2\pi a t}{\lambda} , 
  a \in \cA^\lambda, t \ge 0 
  \right\} . 
\] 
We note  that the set $\cT^\lambda_+$ has both discrete (selection of materials from a finite set) and continuous nature (the physical thickness $t$). Observe that the elements in $\cT^\lambda_+$ have the property that their diagonals have zero imaginary part, and their off-diagonals have zero real part, a property preserved when two elements are multiplied from this set.

\begin{notation}
Let $M \in \mathbb{C}^{2 \times 2}$ be a matrix with the property that $\Im(M_{11})=\Im(M_{22})=0$ and  $\Re(M_{12})=\Re(M_{21})=0$. Then, we will denote a  matrix $\tilde M  \in \mathbb{R}^{2 \times 2}$ by
\begin{equation*}\label{eq:tilde w def}
\tilde M_{i,j} = \begin{cases}
\Re(M_{ij}) & i=j \\
\Im(M_{ij}) & i\neq j
\end{cases},
\text{ for $i,j\in\{1,2\}$. }
\end{equation*}
\end{notation}

We will call the multiplication of transfer matrices as a ``cumulative transfer matrix''. 
For a multi-layer thin film with the cumulative transfer matrix $w\in\mathbb{C}^{2\times2}$ coated on a certain substrate, one can compute the reflectance at wavelength $\lambda$ as
\begin{equation}\label{eq:reflectance}
 R_s^\lambda (\tilde w) := 
 \frac{ (\tilde w_{11}-\Im({\hat a_s^\lambda}) \tilde w_{12}-\Re({\hat a_s^\lambda})\tilde w_{22} )^2 + (\tilde w_{21}+\Im({\hat a_s^\lambda})\tilde w_{22}-\Re({\hat a_s^\lambda) \tilde w_{12}})^2 }
 { (\tilde w_{11}-\Im({\hat a_s^\lambda})\tilde w_{12}+\Re({\hat a_s^\lambda})\tilde w_{22} )^2 + (\tilde w_{21}+\Im({\hat a_s^\lambda})\tilde w_{22}+\Re({\hat a_s^\lambda) \tilde w_{12}})^2 }.
\end{equation}
Note that $ R_s^\lambda (\tilde w)$ is the ratio of two convex quadratic functions in $\tilde w$.

{
Finally, we are ready to formally describe the multi-layer thin films problem: Given a metallic substrate, a set of coating materials $\cM$ and wavelength $\lambda$, decide  coating materials and their thicknesses  to be used in each layer of an $N$-layer thin film such that the reflectance  is maximized.  
}

\subsection{Problem Formulation}
\label{sec:thinFilmsForm}

Using the notation introduced in the previous section, we now formulate multi-layer thin films problem as an instance of the generic model~\eqref{eq:genericB}. In particular, we will set $p$ as the identity matrix, the objective function $f(w)$ as $R_s^\lambda (\tilde w)$ and the set $\cT$ as  $\cT^\lambda_+  $. In the sequel, we will reformulate the problem as an MIQCQP in Section~\ref{sec:thinFilmsFormFinal}, using the structural properties derived in 
Section~\ref{sec:thinFilmsFormPrelim}.

\subsubsection{Some Structural Properties}
\label{sec:thinFilmsFormPrelim}

We will now present some important properties of transfer matrices and the reflectance function.
\begin{proposition}\label{prop:somePropertiesTransfer}
Consider the transfer matrices as defined in~\eqref{eq:transferM}. Then,
\begin{enumerate}[(i)]
\item
$\det (T_{m,t}^\lambda ) = 1$.
\item
$T_{m,t_1}^\lambda T_{m,t_2}^\lambda = T_{m,t_1+ t_2}^\lambda$ for $t_1,t_2\ge0$.
\end{enumerate}
\end{proposition}
\begin{proof}
Statement (i) is clear. Statement (ii) can be checked via straightforward calculation and using trigonometric addition formulas.
\end{proof}

\begin{proposition}\label{prop:reflactanceSimplification}
Consider the reflectance  function as defined  in~\eqref{eq:reflectance} and let $w$ be a cumulative transfer matrix. Then,
\begin{enumerate}[(i)]
\item
$R_s^\lambda(-\tilde w) = R_s^\lambda(\tilde w)$.
\item
$
 R_s^\lambda (\tilde w) = 1-\frac{ 4 \Re({\hat a_s^\lambda}) } { D_s^\lambda (\tilde w)  }
,
$
where
\begin{equation}\label{eq:defDen}
D_s^\lambda (\tilde w) :={  (\tilde w_{11}-\Im({\hat a_s^\lambda})\tilde w_{12})^2 +  (\Re({\hat a_s^\lambda) \tilde w_{12})^2 +
  (\tilde w_{21}+\Im({\hat a_s^\lambda})\tilde w_{22}})^2 + (\Re({\hat a_s^\lambda})\tilde w_{22})^2 + 2\Re({\hat a_s^\lambda}) }.
  \end{equation}
  \end{enumerate}
\end{proposition}
\begin{proof}
Statement (i) is clear.
Statement (ii) can be proven via straightforward algebra and noting that $\det(\tilde w) = 1$.
\end{proof}

Let us now discuss the consequences of the above properties. In particular, we claim that instead of $\cT^\lambda_+$, we can use the family of matrices defined as 
\[
\cT^\lambda :=
\left\{  \begin{bmatrix}
C  &  \ii \frac{S}{a} \\
 \ii {a}{S} & C
\end{bmatrix}: 
  a \in \cA^\lambda, C^2+S^2 = 1, (C,S)\in\mathbb{R}\times\mathbb{R}_+
  \right\}.
\]
This is due to the fact that  $\cT^\lambda_+ = -\cT^\lambda \cup \cT^\lambda$ (Proposition \ref{prop:somePropertiesTransfer}(ii)) and, from optimization point of view, $T_n \in \cT^\lambda$ and $T_n \in \cT_+^\lambda$ are equivalent (Proposition \ref{prop:reflactanceSimplification}(i)). We note that one can recover the physical thickness of a layer associated with a given transfer matrix from $\cT^\lambda$ as
\[
t =  \frac{\lambda \arccos C}{2\pi a} ,
\]
which is well-defined. 
Finally, we remark that the set $\cT^\lambda$ is bounded, which is a property that will  be exploited in the reformulation of the problem provided in the next section.


\subsubsection{Reformulation}
\label{sec:thinFilmsFormFinal}

By utilizing the structural properties derived in the previous section, we can formulate the multi-layer thin film problem as a nonlinear, discrete optimization problem as follows:
\begin{subequations} \label{eq:singleWave}
\begin{align}
\max_{T, w, u, C, S, a} \  &   D_s^\lambda (\tilde w)   \\
\textup{s.t.} 
\ &  \eqref{eq:genericB bilinear},  \eqref{eq:genericB boundary} \notag \\
\ & C_n^2 + S_n^2 = 1, S_n \ge 0 \label{eq:cos sin rel}  \\
\ & (\tilde T_n)_{11} =(\tilde T_n)_{22} = C_n   \label{eq:cos def}  \\
\ & a_n (\tilde T_n)_{12} = (\tilde T_n)_{21} / a_n = S_n  \label{eq:sin def} \\
\ & a_n \in \cA^\lambda.   \label{eq:material pick}
\end{align}
\end{subequations}
Here, constraints  \eqref{eq:cos sin rel}--\eqref{eq:material pick} guarantee that $T_n \in \cT^\lambda$. 

As a final step in the reformulation, we will give a mixed-integer linear representation of  constraints    \eqref{eq:sin def}--\eqref{eq:material pick}. To this end, let us define binary variables $x_{n,m}$, which take value one if material~$m$ is used in layer $n$ and zero otherwise. Moreover, let $v_{n,m}$ be the  auxiliary variables {needed in the disjunction arguments}, representing the quantity $S_n x_{n,m}$. Then, we obtain the following optimization problem with a convex  quadratic maximization objective and a mixed-integer bilinear representable feasible region:
\begin{subequations} \label{eq:singleWaveF}
\begin{align}
\max_{T, w,  u, C, S, v, x} \  &   D_s^\lambda (\tilde w)   \\
\textup{s.t.} 
\ &  \eqref{eq:genericB bilinear},  \eqref{eq:genericB boundary} , \eqref{eq:cos sin rel},   \eqref{eq:cos def}   \notag \\
\ & \sum_{m \in \cM} v_{n,m} =  S_n & n&=1,\dots, N  \\
\ &  \sum_{m \in \cM} v_{n,m} / \hat a_m^\lambda = (\tilde T_n)_{12} & n &=1,\dots, N  \\      
\ &  \sum_{m \in \cM} v_{n,m}  \hat a_m^\lambda = (\tilde T_n)_{21} & n&=1,\dots, N  \\
\ & \sum_{m \in \cM} x_{n,m} = 1 &n&=1,\dots, N  \label{eq:material pick binary}     \\
\ &  0 \le  v_{n,m} \le   x_{n,m},   \ x_{n,m} \in\{0,1\}  & n & =1,\dots, N,\ m\in\cM .
\end{align}
\end{subequations}
We note that the nonconvex MIQCQP~\eqref{eq:singleWaveF} can be solved via Gurobi (version 9) or global solvers such as BARON.


\subsection{Formulation Enhancements}
\label{sec:thinFilmsFormEnhance}

\subsubsection{Bound Tightening}
\label{sec:thinFilmsFormBound}

Since the success of the solution methods of global optimization problems depends on the availability of  variable bounds, we now discuss how to obtain tight variable bounds for problem~\eqref{eq:singleWaveF}. To start with, the following bounds are readily available:
\[
-1 \le C_n, (\tilde T_n)_{11}, (\tilde T_n)_{22} \le 1, 
\ 0 \le S_n \le 1, 0 \le  (\tilde T_n)_{12} \le 1/ \hat   a_L^\lambda, 0 \le  (\tilde T_n)_{21} \le 1/ \hat   a_H^\lambda
  \quad n=1,\dots,N,
\]
where
\begin{equation}\label{eq:highLowMaterials}
\hat   a_L^\lambda := \min \{ \hat a_m^\lambda: m\in\cM \} \text{ and }
\hat   a_H^\lambda := \max \{ \hat a_m^\lambda: m\in\cM \}.
\end{equation}
To obtain the variable bounds for the entries of  the cumulative transfer matrices $u_n$, we will use the following proposition. 
\begin{proposition}\label{prop:uBounds}
Let $\underline \alpha \le \overline \alpha$, $\underline \beta \le \overline \beta$, $ \Gamma := \{\gamma_h\}_{h=1}^H \in \mathbb{R}_+$, and define 
\[
\Phi := \{ (\alpha,\beta,\gamma,C,S) \in \mathbb{R}^5  : \alpha \in [\underline \alpha,  \overline \alpha] ,  \beta \in [\underline \beta,  \overline \beta], \gamma \in \Gamma, C^2 + S^2 = 1, S \ge 0\}.
\]
Then,
\begin{enumerate}[(i)]
\item
$\max\{ \alpha C + \gamma \beta S : (\alpha,\beta,\gamma,C,S) \in \Phi  \} 
= \sqrt{ \max\{\underline \alpha^2, \overline \alpha^2\} + \max\{\gamma^2: \gamma\in\Gamma\}  \max\{0, \overline \beta\}^2 }$.
\item
$\min\{ \alpha C + \gamma \beta S : (\alpha,\beta,\gamma,C,S) \in \Phi  \} 
= -\sqrt{ \max\{\underline \alpha^2, \overline \alpha^2\} + \max\{\gamma^2: \gamma\in\Gamma\}  \max\{0, -\underline \beta\}^2 }$.
\end{enumerate}
\end{proposition}
\begin{proof}
We will prove Statement (i) in three steps. Firstly, 
consider the optimization problem $z^*(\alpha,\beta,\gamma) = \max\{ \alpha C + \gamma \beta S : C^2 + S^2 = 1, S \ge 0 \} $ for some given $(\alpha,\beta,\gamma)\in\mathbb{R}^2\times\mathbb{R}_+$. Then, $z^*(\alpha,\beta,\gamma) $ is equal to $\sqrt{\alpha^2 + (\gamma\beta)^2}$ if $\beta \ge 0$ and $|\alpha|$ otherwise. 
Secondly, consider $z^*(\gamma) := \max\{ \sqrt{\alpha^2 + \max\{0, \gamma\beta\}^2 } : \alpha \in [\underline \alpha,  \overline \alpha] ,  \beta \in [\underline \beta,  \overline \beta] \}$ for some $\gamma\in\mathbb{R}_+$, where the objective function comes from the first step. Then,  
$z^*(\gamma) = \sqrt{  \max\{\underline \alpha^2, \overline \alpha^2\} + \max\{0, \gamma\overline\beta\}^2 }$. Finally, we observe that 
$\max\{ \alpha C + \gamma \beta S : (\alpha,\beta,\gamma,C,S) \in \Phi  \} = \max\{ z^*(\gamma) : \gamma \in \Gamma  \}  $, from which the result follows.

Statement (ii) follows by noting that  $\min\{ \alpha C + \gamma \beta S : (\alpha,\beta,\gamma,C,S) \in \Phi  \} 
= - \max\{ \alpha' C + \gamma \beta' S :  (\alpha',\beta',\gamma,C,S) \in \Phi' \}$ with
$
\Phi' := \{ (\alpha',\beta',\gamma,C,S) \in \mathbb{R}^5  : \alpha' \in [-\overline \alpha,  -\underline \alpha] ,  \beta' \in [-\overline \beta,  -\underline \beta], \gamma \in \Gamma, C^2 + S^2 = 1, S \ge 0\}$, and then applying the previous result.
\end{proof}
Let us now demonstrate how Proposition \ref{prop:uBounds} can be used to derive bounds for the entries of the matrix $u_n$ with an example. Consider one of the constraints of equation \eqref{eq:genericB bilinear}
\[
(\tilde u_n)_{21} =  (\tilde u_{n-1})_{21} C_n +   a_n (\tilde u_{n-1})_{22} S_n ,
\]
for some $n=1,\dots,N$. Since we will proceed recursively and $\tilde u_0 $ is the identity matrix, let us assume that the variable bounds of $(\tilde u_{n-1})_{21}$ and $(\tilde u_{n-1})_{22}$ are available, and denoted as $ [\underline \alpha,  \overline \alpha]$ and $[\underline \beta,  \overline \beta]$, respectively. Also, let $\Gamma = \{a_m^\lambda: m \in \cM\}$. Then, Proposition \ref{prop:uBounds} gives upper and lower bounds for variable $(\tilde u_n)_{21}$. Similar arguments can be used to derive variable bounds for $(\tilde u_n)_{11}$, $(\tilde u_n)_{12}$ and $(\tilde u_n)_{22}$ as well.

%

\subsubsection{Valid Bilinear Equalities}
\label{sec:thinFilmsFormQuadEq}

 Proposition \ref{prop:somePropertiesTransfer}(i) states that the determinant of the transfer matrices is 1, a property preserved by multiplication. Therefore, all the cumulative transfer matrices have determinant 1 as well. In particular, we can add the following bilinear equality to our formulation:
\begin{equation}\label{eq:validQuad}
\det( w) = \tilde w_{1,1} \tilde w_{2,2} + \tilde w_{2,1} \tilde w_{2,1}  = 1.
\end{equation}
Note that, in principle, similar bilinear constraints corresponding to $\det(u_n)=1$ for each $n=1,\dots,N-1$ can be included as well. However, our preliminary experiments have shown that including many such bilinear constraints slows down the solvers. 

\subsubsection{Symmetry Breaking Constraints}
\label{sec:thinFilmsFormSymmBreak}
 Proposition \ref{prop:somePropertiesTransfer}(ii) implies that coating two consecutive layers of the same material with thickness $t_1$ and $t_2$ is equivalent to a single layer of the same materials with thickness $t_1+t_2$. Therefore, including the following inequality, which forbids feasible solutions in which two consequent layers of the same material are used, does not change the optimal value of problem \eqref{eq:singleWaveF}:
 \begin{equation}\label{eq:symmBreak}
x_{n,m} + x_{n+1,m} \le 1 \quad n=1,\dots,N-1, \ m\in\cM.
\end{equation}
However, the above inequality breaks the symmetry in the formulation, hence, is useful in the solution procedure.

\subsection{A Heuristic Approach}
\label{sec:thinFilmsHeur}

A common heuristic approach in the thin films literature to solve problem~\eqref{eq:singleWave} is to use alternating layers of high and low index materials with an optical thickness of a quarter wavelength (see e.g. \citet{macleod2010thin, pedrotti2017introduction}, among others). More precisely, for a given wavelength~$\lambda$, let us denote the refractive  indices of materials among  the set $\cM$ with the highest and lowest values $\hat   a_H^\lambda $ and $\hat   a_L^\lambda $ as computed in \eqref{eq:highLowMaterials}, and
and choose the physical thicknesses as
\[
t_H^\lambda := \frac{\lambda}{4   \hat a_H^\lambda}  \text{ and }
t_L^\lambda :=  \frac{\lambda}{4   \hat a_L^\lambda} .
\]
Consider a feasible solution to  problem \eqref{eq:singleWave} constructed as follows: For each odd (resp. even) index~$n$, we use high (resp. low) index material $H$ (resp. $L$) with thickness $t_H^\lambda$ (resp. $t_L^\lambda)$, that is, 
 the transfer matrix of each layer is chosen as
\begin{equation}\label{eq:thinFilmHeur}
T_n =  T_{H, t_H^\lambda}^\lambda = 
\begin{bmatrix}
0 & \frac\ii{\hat a_H^\lambda} \\
\ii {\hat a_H^\lambda} & 0 
\end{bmatrix} \text{ if $n$ is odd and }
T_n =  T_{L, t_L^\lambda}^\lambda = 
\begin{bmatrix}
0 & \frac\ii{\hat a_L^\lambda} \\
\ii {\hat a_L^\lambda} & 0 
\end{bmatrix} \text{ if $n$ is even}.
\end{equation}
We will now prove that the reflectance of the multi-layer thin films obtained as above converges to~1 as $N\to\infty$. 

\begin{proposition}\label{prop:thinFilmHeur}
Let $\lambda$ be given and consider a feasible solution to problem~\eqref{eq:singleWave} constructed in~\eqref{eq:thinFilmHeur}. Then, 
\[
\lim_{N \to \infty} R_s^\lambda (\tilde u_N) = 1,
\]
where $u_N$ is the corresponding cumulative transfer matrix of an $N$-layer thin films.
\end{proposition}
\begin{proof}
First of all,  the cumulative transfer matrix of an $N$-layer thin film is obtained as
\[
u_N = \begin{cases}
\begin{bmatrix}
0 & \frac{\ii^N}{ {(\hat a_H^\lambda)}^{ \lceil N/2 \rceil } {(\hat a_L^\lambda)}^{ \lfloor N/2 \rfloor } } \\
\ii^N {(\hat a_H^\lambda)}^{ \lceil N/2 \rceil } {(\hat a_L^\lambda)}^{ \lfloor N/2 \rfloor } & 0 
\end{bmatrix} & \text{if $N$ is odd} \\
\\
\begin{bmatrix}
 \frac{\ii^N}{ {(\hat a_H^\lambda)}^{ N/2 } {(\hat a_L^\lambda)}^{ N/2 } } & 0 \\
0 &\ii^N {{(\hat a_H^\lambda)}^{  N/2  } {(\hat a_L^\lambda)}^{  N/2  }}
\end{bmatrix} & \text{if $N$ is even}
\end{cases}.
\]
Then, we have
\[
D_s^\lambda(\tilde u_N) = \begin{cases}
\frac{|\hat a_s^\lambda|^2}{ \left({(\hat a_H^\lambda)}^{ \lceil N/2 \rceil } {(\hat a_L^\lambda)}^{ \lfloor N/2 \rfloor }\right)^2 } +  { \left({(\hat a_H^\lambda)}^{ \lceil N/2 \rceil } {(\hat a_L^\lambda)}^{ \lfloor N/2 \rfloor }\right)^2 } + 2\Re({\hat a_s^\lambda})  & \text{if $N$ is odd} \\
\frac{|\hat a_s^\lambda|^2}{ \left({(\hat a_H^\lambda)}^{ N/2 } {(\hat a_L^\lambda)}^{ N/2 }\right)^2 } +  { \left({(\hat a_H^\lambda)}^{ N/2 } {(\hat a_L^\lambda)}^{ N/2 }\right)^2 } + 2\Re({\hat a_s^\lambda})  & \text{if $N$ is even} 
\end{cases},
\]
where $D_s^\lambda(\cdot) $ is defined as in \eqref{eq:defDen}. Note that since $\hat a_H^\lambda > 1$ and  $\hat a_L^\lambda>1$, we have that $ \lim_{ N \to \infty} D_s^\lambda(\tilde u_N)  = \infty$. Hence, we conclude that $\lim_{N \to \infty} R_s^\lambda (\tilde u_N) = 1$ due to Proposition \ref{prop:reflactanceSimplification}(ii).
\end{proof}

Proposition \ref{prop:thinFilmHeur} justifies the use of the heuristic approach introduced above, especially when a large number of layers are allowed to be used. However, in practice,  thin films with a small number of layers can be preferred due to cost considerations and manufacturing challenges. 
In such cases, the heuristic solutions obtained may not be optimal, as demonstrated by our computational experiments presented in the next section, and our optimization-based approach might prove very useful.


\subsection{Computations}
\label{sec:thinFilmsComp}

In this section,  we present our extensive computations  for the multi-layer thin films problem. In this analysis, we use the coating materials  \ch{SiO2}, \ch{TiO2},  \ch{MgF2}, \ch{Al2O3}, and metallic substrates Tungsten, Tantalum, Molybdenum, Niobium. 
The necessary refractive  index data is obtained from \citet{keccebacs2018enhancing}, which is gathered from multiple sources \citep{malitson1965interspecimen,palik1998handbook,dodge1984refractive,dodge2refractive,golovashkin1969optical}. We utilize BARON 19.12.7 and Gurobi 9 to solve the nonconvex MIQCQP~\eqref{eq:singleWaveF} on a 64-bit personal computer with Intel Core i7
CPU 2.60GHz processor (16 GB RAM). The relative optimality gap is set to  0.001 for all experiments.

\subsubsection{Computational Efficiency}
\label{sec:thinFilmsCompEff}

We first carry out some preliminary experiments to decide between two competing solvers BARON and Gurobi to solve
 problem~\eqref{eq:singleWaveF}, and to demonstrate the effect of valid bilinear equalities~\eqref{eq:validQuad} and symmetry breaking constraints~\eqref{eq:symmBreak}. Our results  presented in Table~\ref{tab:thinFilmsCompEff} clearly show that Gurobi is the faster solver for this problem by at least one order-of-magnitude. We also observe that the addition of constraints~\eqref{eq:validQuad} and~\eqref{eq:symmBreak} help improve the computational performance of both solvers significantly (except $N=2$ for BARON).

\begin{table}[H]\small
\caption{Computational times (in seconds) of different methods to solve problem~\eqref{eq:singleWaveF} for different wavelengths $\lambda$ (in nanometers) and number of layers $N$ on a Tungsten substrate. ``Enh.'' stands for ``Enhanced'' and refers to the addition of 
equations~\eqref{eq:validQuad} and~\eqref{eq:symmBreak}.}\label{tab:thinFilmsCompEff}
\centering
\begin{tabular}{c|rr|rr|rrr|rrr}
   & \multicolumn{ 2}{c|}{BARON} & \multicolumn{ 2}{c|}{BARON Enh.} &          \multicolumn{ 3}{c|}{Gurobi} &         \multicolumn{ 3}{c}{Gurobi Enh.} \\
$\lambda$ (nm) &       $ N=2$ &        $N=3 $&      $  N=2$ &      $  N=3$ &      $  N=2$ &        $N=3$ &       $ N=4$ &      $  N=2$ &       $ N=3$ &      $  N=4$ \\
\hline
       450 &       0.92 &      84.26 &       9.86 &      82.03 &       0.39 &       5.36 &      33.05 &       0.85 &       3.92 &      14.41 \\

       600 &       1.03 &     152.28 &       1.24 &     154.89 &       0.45 &       7.96 &      58.79 &       0.45 &       4.73 &      24.16 \\

       750 &       1.10 &     225.27 &       5.59 &     187.77 &       0.43 &      11.05 &      68.53 &       0.60 &       5.50 &      26.56 \\

       900 &       1.03 &     221.17 &       3.92 &     181.93 &       0.54 &       9.47 &      81.15 &       0.55 &       5.03 &      29.75 \\

      1200 &       0.98 &     245.74 &       2.75 &     133.05 &       0.44 &       9.30 &      73.91 &       0.58 &       4.82 &      31.41 \\

      1500 &       0.94 &     217.62 &       4.49 &     182.65 &       0.42 &       9.17 &      86.32 &       0.47 &       5.55 &      29.62 \\

      1800 &       0.98 &     283.27 &       7.56 &     225.14 &       0.48 &       8.81 &      78.24 &       0.44 &       5.80 &      33.36 \\

      2100 &       0.74 &     314.47 &       4.10 &     144.45 &       0.40 &      10.61 &      88.68 &       0.59 &       6.53 &      32.16 \\

      2400 &       0.76 &     449.25 &       5.40 &     180.88 &       0.42 &      13.08 &      87.57 &       0.59 &       6.68 &      36.68 \\
\hline
   Avg. &       0.94 &     243.70 &       4.99 &     163.64 &       0.44 &       9.42 &      72.92 &       0.57 &       5.40 &      28.68 \\

\end{tabular}  
\end{table}

As a result of these preliminary experiments, we have decided to use Gurobi with enhancements in the detailed experiments presented below.

\subsubsection{Comparison with the Heuristic Approach}
\label{sec:thinFilmsCompHeur}

We now compare the solutions obtained from the heuristic approach in Section \ref{sec:thinFilmsHeur} and solving the optimization problem \eqref{eq:singleWaveF}. In Tables \ref{tab:compareTungsten}--\ref{tab:compareNiobium}, we report the reflectance values obtained for different wavelengths $\lambda$ and  number of layers $N$ on four metallic substrates.
In these computational experiments, we pre-terminate Gurobi once a feasible solution with a reflectance  of at least 0.995 is obtained (this is enforced via the use of the parameter {\texttt BestObjStop}).

\begin{table}[H]\scriptsize
\caption{Comparison of the reflectance values obtained by problem \eqref{eq:singleWaveF} and the heuristic approach for different wavelengths $\lambda$ and  number of layers $N$ on a Tungsten substrate (``AT (s)'' stands for ``average time in seconds'').}\label{tab:compareTungsten}
\centering
\begin{tabular}{cr|rrrrrr|rrrrr}

   &            &                                              \multicolumn{ 6}{c|}{Heuristic} &                                   \multicolumn{ 5}{c}{Optimal} \\

    $\lambda$ (nm) &        $N=0$ &        $N=1$ &        $N=2$ &        $N=3$ &        $N=4$ &        $N=5$ &        $N=6$ &        $N=1$ &        $N=2$ &        $N=3$ &        $N=4$ &        $N=5$ \\
\hline
       450 &      0.470 &      0.279 &      0.865 &      0.778 &      0.973 &      0.953 &      0.995 &      0.553 &      0.870 &      0.894 &      0.974 &      0.979 \\

       600 &      0.508 &      0.209 &      0.857 &      0.683 &      0.966 &      0.917 &      0.992 &      0.563 &      0.862 &      0.879 &      0.967 &      0.971 \\

       750 &      0.500 &      0.169 &      0.846 &      0.633 &      0.961 &      0.896 &      0.990 &      0.545 &      0.851 &      0.866 &      0.962 &      0.966 \\

       900 &      0.521 &      0.223 &      0.850 &      0.661 &      0.961 &      0.903 &      0.990 &      0.579 &      0.856 &      0.875 &      0.962 &      0.968 \\

      1200 &      0.642 &      0.283 &      0.892 &      0.660 &      0.972 &      0.899 &      0.993 &      0.683 &      0.897 &      0.909 &      0.973 &      0.976 \\

      1500 &      0.698 &      0.384 &      0.910 &      0.718 &      0.976 &      0.917 &      0.994 &      0.740 &      0.914 &      0.926 &      0.977 &      0.980 \\

      1800 &      0.866 &      0.616 &      0.962 &      0.805 &      0.990 &      0.942 &      0.997 &      0.881 &      0.964 &      0.967 &      0.990 &      0.991 \\

      2100 &      0.933 &      0.751 &      0.981 &      0.844 &      0.995 &      0.951 &      0.999 &      0.938 &      0.982 &      0.983 &      0.995 &    0.995         \\

      2400 &      0.951 &      0.787 &      0.986 &      0.831 &      0.996 &      0.942 &      0.999 &      0.953 &      0.986 &      0.987 &      0.996 &     0.996        \\
\hline
AT (s) &            &            &            &            &            &            &            &       0.28 &       0.57 &       5.40 &      23.85 &    5517.26      \\

\end{tabular}  
\end{table}

\begin{table}[H]\scriptsize
\caption{Comparison of the reflectance values obtained by problem \eqref{eq:singleWaveF} and the heuristic approach for different wavelengths $\lambda$ and  number of layers $N$ on a Tantalum substrate.}\label{tab:compareTantalum}
\centering
\begin{tabular}{cr|rrrrrr|rrrrr}

   &            &                                              \multicolumn{ 6}{c|}{Heuristic} &                                   \multicolumn{ 5}{c}{Optimal} \\

    $\lambda$ (nm) &        $N=0$ &        $N=1$ &        $N=2$ &        $N=3$ &        $N=4$ &        $N=5$ &        $N=6$ &        $N=1$ &        $N=2$ &        $N=3$ &        $N=4$ &        $N=5$ \\
\hline
       450 &      0.409 &      0.329 &      0.842 &      0.805 &      0.968 &      0.960 &      0.994 &      0.530 &      0.850 &      0.887 &      0.969 &  0.977          \\

       600 &      0.361 &      0.397 &      0.787 &      0.807 &      0.947 &      0.953 &      0.988 &      0.548 &      0.809 &      0.874 &      0.953 & 0.970           \\

       750 &      0.672 &      0.592 &      0.903 &      0.866 &      0.976 &      0.966 &      0.994 &      0.772 &      0.915 &      0.940 &      0.979 &   0.985         \\

       900 &      0.814 &      0.663 &      0.948 &      0.878 &      0.987 &      0.968 &      0.997 &      0.856 &      0.953 &      0.962 &      0.988 &    0.991        \\

      1200 &      0.914 &      0.751 &      0.977 &      0.889 &      0.994 &      0.970 &      0.999 &      0.925 &      0.978 &      0.981 &      0.994 &   0.995         \\

      1500 &      0.951 &      0.813 &      0.987 &      0.895 &      0.997 &      0.970 &      0.999 &      0.955 &      0.987 &      0.988 &      0.996 &   0.995         \\

      1800 &      0.963 &      0.835 &      0.990 &      0.882 &      0.997 &      0.963 &      0.999 &      0.965 &      0.990 &      0.991 &      0.997 &   0.996         \\

      2100 &      0.970 &      0.851 &      0.992 &      0.866 &      0.998 &      0.955 &      0.999 &      0.971 &      0.992 &      0.992 &      0.998 &   0.995         \\

      2400 &      0.973 &      0.860 &      0.992 &      0.848 &      0.998 &      0.943 &      0.999 &      0.974 &      0.992 &      0.993 &      0.997 &   0.996         \\
\hline
AT (s) &            &            &            &            &            &            &            &       0.28 &       0.51 &       5.19 &      16.74 &  1478.59          \\

\end{tabular}  
\end{table}

\begin{table}[H]\scriptsize
\caption{Comparison of the reflectance values obtained by problem \eqref{eq:singleWaveF} and the heuristic approach for different wavelengths $\lambda$ and  number of layers $N$ on a Molybdenum substrate.}\label{tab:compareMolybdenum}
\centering
\begin{tabular}{cr|rrrrrr|rrrrr}

   &            &                                              \multicolumn{ 6}{c|}{Heuristic} &                                   \multicolumn{ 5}{c}{Optimal} \\

    $\lambda$ (nm) &        $N=0$ &        $N=1$ &        $N=2$ &        $N=3$ &        $N=4$ &        $N=5$ &        $N=6$ &        $N=1$ &        $N=2$ &        $N=3$ &        $N=4$ &        $N=5$ \\
\hline
       450 &      0.569 &      0.325 &      0.896 &      0.791 &      0.979 &      0.956 &      0.996 &      0.643 &      0.901 &      0.920 &      0.980 &   0.984         \\

       600 &      0.567 &      0.218 &      0.878 &      0.676 &      0.971 &      0.915 &      0.993 &      0.613 &      0.882 &      0.896 &      0.972 &   0.976         \\

       750 &      0.566 &      0.191 &      0.872 &      0.631 &      0.968 &      0.895 &      0.992 &      0.607 &      0.875 &      0.888 &      0.969 &      0.972      \\

       900 &      0.570 &      0.261 &      0.868 &      0.677 &      0.966 &      0.908 &      0.991 &      0.626 &      0.874 &      0.892 &      0.967 &     0.972       \\

      1200 &      0.786 &      0.492 &      0.939 &      0.768 &      0.984 &      0.934 &      0.996 &      0.814 &      0.942 &      0.949 &      0.985 &      0.987      \\

      1500 &      0.890 &      0.638 &      0.970 &      0.806 &      0.992 &      0.943 &      0.998 &      0.900 &      0.971 &      0.973 &      0.993 &      0.993      \\

      1800 &      0.935 &      0.735 &      0.982 &      0.824 &      0.995 &      0.945 &      0.999 &      0.939 &      0.983 &      0.984 &      0.995 &      0.995      \\

      2100 &      0.958 &      0.804 &      0.988 &      0.837 &      0.997 &      0.946 &      0.999 &      0.960 &      0.988 &      0.989 &      0.995 &     0.995       \\

      2400 &      0.969 &      0.844 &      0.991 &      0.840 &      0.998 &      0.941 &      0.999 &      0.970 &      0.991 &      0.992 &      0.997 &      0.995      \\
\hline
AT (s) &            &            &            &            &            &            &            &       0.25 &       0.44 &       5.51 &      18.64 &    5717.69        \\

\end{tabular}  
\end{table}

\begin{table}[H]\scriptsize
\caption{Comparison of the reflectance values obtained by problem \eqref{eq:singleWaveF} and the heuristic approach for different wavelengths $\lambda$ and  number of layers $N$ on a Niobium substrate.}\label{tab:compareNiobium}
\centering
\begin{tabular}{cr|rrrrrr|rrrrr}

   &            &                                              \multicolumn{ 6}{c|}{Heuristic} &                                   \multicolumn{ 5}{c}{Optimal} \\

    $\lambda$ (nm) &        $N=0$ &        $N=1$ &        $N=2$ &        $N=3$ &        $N=4$ &        $N=5$ &        $N=6$ &        $N=1$ &        $N=2$ &        $N=3$ &        $N=4$ &        $N=5$ \\
\hline
       450 &      0.558 &      0.486 &      0.890 &      0.862 &      0.978 &      0.972 &      0.996 &      0.688 &      0.900 &      0.931 &      0.980 &   0.987         \\

       600 &      0.573 &      0.387 &      0.878 &      0.785 &      0.971 &      0.947 &      0.993 &      0.663 &      0.887 &      0.912 &      0.973 &   0.979         \\

       750 &      0.620 &      0.407 &      0.888 &      0.775 &      0.972 &      0.940 &      0.993 &      0.696 &      0.896 &      0.917 &      0.974 &     0.979       \\

       900 &      0.726 &      0.485 &      0.922 &      0.794 &      0.980 &      0.944 &      0.995 &      0.775 &      0.927 &      0.939 &      0.981 &     0.985       \\

      1200 &      0.875 &      0.641 &      0.966 &      0.831 &      0.991 &      0.952 &      0.998 &      0.890 &      0.967 &      0.971 &      0.992 &    0.993        \\

      1500 &      0.924 &      0.717 &      0.980 &      0.836 &      0.995 &      0.952 &      0.999 &      0.930 &      0.980 &      0.982 &      0.995 &    0.995        \\

      1800 &      0.941 &      0.739 &      0.984 &      0.805 &      0.996 &      0.938 &      0.999 &      0.944 &      0.984 &      0.985 &      0.995 &     0.995       \\

      2100 &      0.953 &      0.775 &      0.987 &      0.792 &      0.997 &      0.927 &      0.999 &      0.955 &      0.987 &      0.988 &      0.996 &    0.996        \\

      2400 &      0.952 &      0.760 &      0.986 &      0.733 &      0.996 &      0.894 &      0.999 &      0.953 &      0.986 &      0.987 &      0.996 &    0.996        \\
\hline
AT (s) &            &            &            &            &            &            &            &       0.25 &       0.49 &       5.18 &      18.75 &   2297.35        \\

\end{tabular}  
\end{table}

%
%
%
%
%
%

We have similar observations for different metallic substrates. Firstly, we report that the success of the heuristic method heavily depends on the parity of the number of layers $N$. If $N$ is even, then the solution of the heuristic method performs quite well compared to the optimal solutions. However, if $N$ is odd, then the performance of the heuristic solution can be significantly worse than the optimal solution. Interestingly, in the case $N$ is odd, the heuristic solution obtained with $N-1$ layers seems to be even better. Also, the reflectance value difference between the heuristic and optimal solutions gets smaller as $N$ increases as expected due to Proposition~\ref{prop:thinFilmHeur}. Therefore, our proposed optimization approach is especially useful when $N$ is odd or small. 
Since thin films with smaller number of layers providing high reflectance are desirable in practice, our approach can be beneficial in such cases. 
We remark that we have not solved the optimization model for $N \ge 6$ since the heuristic method already provides very high reflectance values, which are sufficient from an application point of view when $N=6$.

We note that optimal solutions have similar characteristic to the solutions obtained by the heuristic method. In particular, the ordering of materials are again alternating between the highest and lowest indexed coating materials (in this case, \ch{TiO2} and \ch{MgF2}, respectively). However, the optimal optical thicknesses of the layers are not necessarily equal to quarter wavelength. We leave the formal analysis and exploration of these observations as future work. 
Finally, we notice that Gurobi is quite successful in finding high quality solutions early, and spends most of the effort in improving the dual bound to certify that the feasible solution obtained is, in fact, optimal. 




%
%
%
%
%

\section{An Application from Biology: Antibiotics Time Machine}
\label{sec:antibiotic}{
In this section, we  study the antibiotics time machine problem from biology. After providing a formal problem definition in Section \ref{sec:antibioticDef}, we present an MILP formulation in Section \ref{sec:antibioticForm} by utilizing the linearization approach derived in Section \ref{sec:mainLinear}. Finally, we present the results of our computational experiments in Section \ref{sec:antibioticComp} using real and synthetic datasets. 
}

\subsection{Problem Definition}
\label{sec:antibioticDef}

Let us first fix the notation used in this section. 
Consider a string  $s$ of size $g$ with $s_i\in\{0,1\}$, $i=1,\dots,g$. Here, each  string $s \in \{0,1\}^g$  represents a bacterial genotype  with $g$  alleles, and each character $s_i$  indicates whether there is a mutation in the corresponding allele ($s_i=1$) or not ($s_i=0$). The genotype $s=0$ is special and it is called the \textit{wild type} since it has no mutations.
Let us denote the total number of states as $d:=2^g$.
Suppose that we have $K$ antibiotics and the transition between genotypes (or states) under the administration of antibiotic $k$ is governed by a probability matrix $\hat T_k \in \mathbb{R}^{d \times d}$, $k=1,\dots,K$. In other words, there is an associated Markov chain for each drug.  Under the common assumption of Strong Selection Weak Mutation (SSWM) developed in \citet{gillespie1983simple, gillespie1984molecular}, only the transition probability between two genotypes $s$ and $s'$ that have exactly one different character (i.e., $\|s-s'\|_1 = 1$) can be  positive. Let us call the set of such pairs as $\cN$.

The antibiotics time machine problem is formally described as follows: Given $K$ drugs and the initial genotype, find a treatment plan of length $N$ such that the probability of reaching the wild type is maximized.  

Let us use Figure  \ref{fig:antibiotic} as an example to explain the notation and clarify the problem setting. In this illustration, we have $g=3$ alleles and $K=2$ drugs, Blue (solid arcs) and Red (dotted arcs). Suppose that we are seeking  a treatment plan of length $N=3$ given the initial genotype $111$.
 Notice that there is no path with positive probability in the Markov chain corresponding to a single drug going from the initial genotype to the wild type in three steps. However, administering Blue, Blue, Red or 
Red, Blue, Red drugs in sequence both provide positive probabilities. Then, our aim is to decide which of these treatment plans have the highest probability (observe that these are the only treatment plans with positive probabilities for this instance), which involves a series of matrix multiplications.  
We would like to note that the antibiotics time machine problem can be seen as solving a ``static'' Markov decision process in which all the decisions are made before any realizations become available. 

%
%
%
%
%
%
%

\def\layersep{4.5cm}
\def\nodesep{3.5}

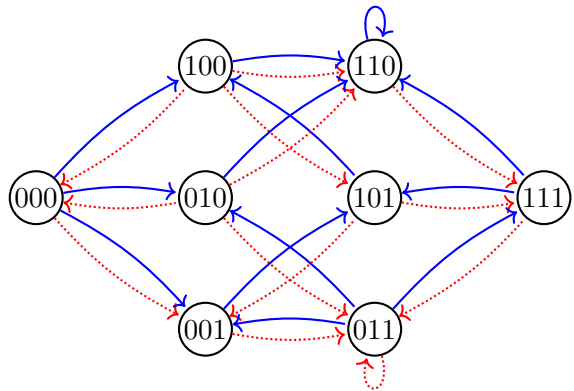
\begin{figure}[H]
\centering
\begin{tikzpicture}[shorten >=1pt,->,draw=black!50, node distance=\layersep, thick, scale=0.5]
    \tikzstyle{every pin edge}=[<-,shorten <=1pt]
    \tikzstyle{neuron}=[circle,fill=black!75,minimum size=20pt,draw=black,inner sep=0pt]
    \tikzstyle{input neuron}=[neuron, fill=green!0];
    \tikzstyle{output neuron}=[neuron, fill=red!0, minimum size=16pt,draw=white];
    \tikzstyle{hidden neuron}=[neuron, fill=blue!50];
    \tikzstyle{annot} = [text width=4em, text centered]

	\node[input neuron] (I000) at (0,0) {{$000$}};

        
        	\node[input neuron] (I100) at (\layersep,\nodesep) {{$100$}};
        	\node[input neuron] (I010) at (\layersep,0) {{$010$}};
        	\node[input neuron] (I001) at (\layersep,-\nodesep) {{$001$}};
        

%
            
        	\node[input neuron] (I110) at (2*\layersep,\nodesep) {{$110$}};
        	\node[input neuron] (I101) at (2*\layersep,0) {{$101$}};
        	\node[input neuron] (I011) at (2*\layersep,-\nodesep) {{$011$}};


%
%
%
%
%
%
%
%
            
	\node[input neuron] (I111) at (3*\layersep,0) {{$111$}};

%


%

    \path [<-] (I000) edge [bend right=10,red,densely dotted] (I100);
    \path [<-] (I000) edge [bend right=10,red,densely dotted] (I010);
    \path [->] (I000) edge [bend right=10,red,densely dotted] (I001);
    \path [->] (I000) edge [bend left=10,blue] (I100);
    \path [->] (I000) edge [bend left=10,blue] (I010);
    \path [->] (I000) edge [bend left=10,blue] (I001);   
    
    \path [->] (I100) edge [bend right=10,red,densely dotted] (I110);
    \path [->] (I100) edge [bend right=10,red,densely dotted] (I101);
    \path [->] (I100) edge [bend left=10,blue] (I110);
    \path [<-] (I100) edge [bend left=10,blue] (I101);
    
    \path [->] (I010) edge [bend right=10,red,densely dotted] (I110);   
    \path [->] (I010) edge [bend right=10,red,densely dotted] (I011);    
    \path [->] (I010) edge [bend left=10,blue] (I110);   
    \path [<-] (I010) edge [bend left=10,blue] (I011);    

    \path [<-] (I001) edge [bend right=10,red,densely dotted] (I101);
    \path [->] (I001) edge [bend right=10,red,densely dotted] (I011);
    \path [->] (I001) edge [bend left=10,blue] (I101);
    \path [<-] (I001) edge [bend left=10,blue] (I011);
        
    \path [->] (I110) edge [bend right=10,red,densely dotted] (I111);    
    \path [<-] (I110) edge [bend left=10,blue] (I111);    
        \path [<-] (I110) edge [loop above,blue] (I110);    
        
    \path [->] (I101) edge [bend right=10,red,densely dotted] (I111);
    \path [<-] (I101) edge [bend left=10,blue] (I111);
    
    \path [<-] (I011) edge [bend right=10,red,densely dotted] (I111);
        \path [<-] (I011) edge [loop below,red,densely dotted] (I011);
    \path [->] (I011) edge [bend left=10,blue] (I111);
                
\end{tikzpicture}
\caption{Illustration of an antibiotics time machine problem instance with $g=3$ alleles and $K=2$ drugs. Probabilistic state transitions under each drug are represented by the arcs.}\label{fig:antibiotic}
\end{figure}

\subsection{Problem Formulation}
\label{sec:antibioticForm}

We will now formulate antibiotics time machine problem as an instance of the generic model~\eqref{eq:MILP} by benefiting from the fact that $\cT$ is a finite set.
Recall that we denote the probability transition matrix of drug $k$ as $\hat T_k \in \mathbb{R}^{d \times d}$. Then, we have 
$
\cT = \{\hat T_k: k =1,\dots,K \}
$.
Let $p$ and $q$ be the unit row vectors corresponding to the initial and final states, respectively. Then, the objective function $f(w):=w q^T$ gives the probability of reaching to the final state in exactly $N$ steps with the decisions $T_1, \dots, T_N$. Since the objective function is linear, antibiotics time machine problem can be solved as an MILP via~\eqref{eq:MILP}. In this formulation,  we  select the outer-approximating polytopes $\bar\cU_n$ as the  standard simplex of order $d$, that is,  
\[
\bar\cU_n = \Delta_d := \bigg \{ u \in \mathbb{R}_+^{d} : \ \sum_{j=1}^d u_{j} = 1  \bigg\},
\]
for $n=1,\dots,N$, as each variable row vector $u_n$  corresponds to the probability distribution after administering the selected first~$n$ drugs.

\subsection{Computations}
\label{sec:antibioticComp}

In this section, we present the computational results obtained by solving the antibiotics time machine problem using a real dataset from  \citet{mira2015rational} in Section~\ref{sec:antibioticCompReal} and a synthetic dataset in Section~\ref{sec:antibioticCompSynthetic}. We compare the computational effort of complete enumeration and  the MILP~\eqref{eq:MILP}, and discuss their scalability issues.


\subsubsection{A Real Dataset}
\label{sec:antibioticCompReal}

We use the experimental growth data and probability calculations from  \citet{mira2015rational}  to obtain the transition probability matrices. 
In particular, let $\omega_{k,j}$ be the  growth rate of genotype $j$ under antibiotic $k$. The main principle behind the probability calculation   is that if $\omega_{k,j'} > \omega_{k,j}$, then $\hat T_{k,(j,j')} > 0 $ for $(j,j') \in \cN$. 
 Two different probability models are used in  \citet{mira2015rational}.
\begin{itemize}
\item
Correlated Probability Model (CPM):  In this model, the probabilities are computed as
\begin{equation*}\label{eq:probCPM}
\hat T^{\text{\tiny CPM}}_{k,(j,j')} = \frac{ \max\{0, \omega_{k,j'} - \omega_{k,j} \} }{ \sum_{ j'':(j'',j) \in \cN}  \max\{0, \omega_{k,j''} - \omega_{k,j} \} }, \quad (j,j') \in\cN.
\end{equation*}
Here, we use the convention $\frac{0}{0}=0$.

\item
Equal Probability Model (EPM): In this model, the probabilities are computed as
\begin{equation*}\label{eq:probEPM}
\hat T^{\text{\tiny EPM}}_{k,(j,j')} = \frac{ \mathbbm{1} ( \omega_{k,j'} > \omega_{k,j} ) }{ \sum_{ j'':(j'',j) \in \cN}  \mathbbm{1} ( \omega_{k,j''} > \omega_{k,j}) }, \quad (j,j') \in\cN.
\end{equation*}
Here, $\mathbbm{1}(\cdot)$ denotes the indicator function.
\end{itemize}
In both models, we set $\hat T_{k,(j,j)} = 1$ if genotype $j$ is an absorbing state under drug $k$, that is, $ \omega_{k,j} > \omega_{k,j'}$ for each $j'$ such that $(j,j') \in\cN$.
In this real dataset, the measurements of  $K=15$ drugs are  reported for genotypes with $g=4$ alleles, that is, for $d=16$ states.

We now compare the computational performance of solving the antibiotics time machine problem with complete enumeration (which is the method used in \citet{mira2015rational} for up to $N=6$) versus MILP~\eqref{eq:MILP}  in 
Tables~\ref{tab:MiraCPM} and \ref{tab:MiraEPM}. 
 In these tables, we report i) the maximum probability of going from each initial genotype to the wild type (0000), ii)  the average computation time  and the number of branch-and-bound nodes (BBNode)  under the absolute optimality gap of 0.001, and iii) the average computation time of complete enumeration (Enum) up to $N=6$ and its estimates for larger values. 
We use 
Gurobi 9 as the MILP solver on a 64-bit personal computer with Intel Core i7
CPU 2.60GHz processor (16 GB RAM).

\begin{landscape}
\begin{table}[H]\scriptsize
\caption{Maximum probabilities, average number of branch-and-bound nodes and  average run times (in seconds) using the data from \citet{mira2015rational}  under  CPM.}\label{tab:MiraCPM}
\centering
\begin{tabular}{c|rrrrrr|rrrrrrrrr}

 initial          &          $N=1$ &        $N=2$ &        $N=3$ &        $N=4$ &        $N=5$ &        $N=6$ &        $N=7$ &        $N=8$ &        $N=9$ &       $N=10$      & $N=11$ &       $N=12$ &  $N=13$ &        $N=14$ &        $N=15$  \\
\hline
1000 &      1.000 &      1.000 &      1.000 &      1.000 &      1.000 &      1.000 &      1.000 &      1.000 &      1.000 &      1.000 &      1.000 &      1.000 &     1.000 &      1.000 &      1.000 \\

0100 &      0.617 &      0.617 &      0.617 &      0.617 &      0.617 &      0.617 &      0.617 &      0.617 &      0.617 &      0.617 &      0.617 &      0.617   &   0.617 &      0.617 &      0.617 \\

0010 &      0.715 &      0.715 &      0.715 &      0.715 &      0.715 &      0.715 &      0.715 &      0.715 &      0.715 &      0.715 &      0.715 &      0.715     & 0.715 &      0.715 &      0.715 \\

0001 &      0.287 &      0.287 &      0.592 &      0.592 &      0.726 &      0.726 &      0.729 &      0.729 &      0.729 &      0.729 &      0.731 &      0.731 &     0.732 &      0.732 &      0.733 \\

1100 &      0.000 &      0.617 &      0.617 &      0.617 &      0.617 &      0.617 &      0.617 &      0.617 &      0.617 &      0.617 &      0.617 &      0.617 &     0.617 &      0.617 &      0.617 \\

1010 &      0.000 &      0.715 &      0.715 &      0.715 &      0.715 &      0.715 &      0.715 &      0.715 &      0.715 &      0.715 &      0.715 &      0.715  &    0.715 &      0.715 &      0.715 \\

1001 &      0.000 &      0.559 &      0.559 &      0.726 &      0.726 &      0.729 &      0.729 &      0.729 &      0.729 &      0.731 &      0.731 &      0.732 &     0.732 &      0.733 &      0.733 \\

0110 &      0.000 &      0.617 &      0.617 &      0.617 &      0.617 &      0.617 &      0.617 &      0.617 &      0.617 &      0.617 &      0.617 &      0.617  &    0.617 &      0.617 &      0.617 \\

0101 &      0.000 &      0.592 &      0.592 &      0.612 &      0.612 &      0.617 &      0.617 &      0.617 &      0.617 &      0.617 &      0.617 &      0.617   &   0.617 &      0.617 &      0.617 \\

0011 &      0.000 &      0.361 &      0.361 &      0.586 &      0.600 &      0.617 &      0.617 &      0.617 &      0.617 &      0.617 &      0.617 &      0.617  &    0.617 &      0.617 &      0.617 \\

1110 &      0.000 &      0.000 &      0.617 &      0.617 &      0.617 &      0.617 &      0.617 &      0.617 &      0.617 &      0.617 &      0.617 &      0.617 &     0.617 &      0.617 &      0.617 \\

1101 &      0.000 &      0.000 &      0.592 &      0.592 &      0.617 &      0.617 &      0.617 &      0.617 &      0.617 &      0.617 &      0.617 &      0.617 &     0.617 &      0.617 &      0.617 \\

1011 &      0.000 &      0.000 &      0.532 &      0.532 &      0.684 &      0.690 &      0.691 &      0.693 &      0.694 &      0.694 &      0.694 &      0.695 &     0.696 &      0.697 &      0.697 \\

0111 &      0.000 &      0.000 &      0.586 &      0.600 &      0.617 &      0.617 &      0.617 &      0.617 &      0.617 &      0.617 &      0.617 &      0.617 &     0.617 &      0.617 &      0.617 \\

1111 &      0.000 &      0.000 &      0.000 &      0.617 &      0.617 &      0.617 &      0.617 &      0.617 &      0.617 &      0.617 &      0.617 &      0.617 &     0.617 &      0.617 &      0.617 \\
\hline
MILP (s) &       0.05 &       0.07 &       0.10 &       0.12 &       0.19 &       0.41 &       0.63 &       1.01 &       1.42 &       2.01 &       4.30 &      10.21 &     23.02 &      48.67 &     147.29 \\

BBNode &       0.00 &       0.27 &       0.87 &       1.73 &       8.53 &      26.53 &      84.33 &     252.53 &     602.47 &    1385.53 &    3628.40 &    8093.60 &  21219.20 &   42344.73 &   97275.80 \\

Enum (s) &       0.00 &       0.01 &       0.06 &       0.94 &      16.46 &     272.12 &   $4.1 \cdot 10^3$ &   $6.1\cdot10^4$ &  
$ 9.2 \cdot10^5$  &   $1.4\cdot10^7$ & $ 2.1\cdot10^8$  &   $3.1\cdot10^9$ &   $4.7\cdot10^{10}$ & $ 7.0\cdot10^{11}$  &   $1.1\cdot10^{13}$    \\

\end{tabular}  
\end{table}

\begin{table}[H]\scriptsize
\caption{Maximum probabilities, average number of branch-and-bound nodes and  average run times (in seconds) using the data from \citet{mira2015rational}  under  EPM.}\label{tab:MiraEPM}
\centering
\begin{tabular}{c|rrrrrr|rrrrrrrrr}

 initial          &          $N=1$ &        $N=2$ &        $N=3$ &        $N=4$ &        $N=5$ &        $N=6$ &        $N=7$ &        $N=8$ &        $N=9$ &       $N=10$      & $N=11$ &       $N=12$ &  $N=13$ &        $N=14$ &        $N=15$  \\
\hline
1000 &      1.000 &      1.000 &      1.000 &      1.000 &      1.000 &      1.000 &      1.000 &      1.000 &      1.000 &      1.000 &      1.000 &      1.000            &      1.000 &      1.000 &      1.000 \\

0100 &      0.333 &      0.333 &      0.333 &      0.375 &      0.458 &      0.458 &      0.463 &      0.463 &      0.471 &      0.479 &      0.479 &      0.515            &      0.515 &      0.520 &      0.520 \\

0010 &      0.500 &      0.500 &      0.500 &      0.500 &      0.500 &      0.500 &      0.512 &      0.512 &      0.515 &      0.516 &      0.520 &      0.520            &      0.526 &      0.526 &      0.532 \\

0001 &      0.500 &      0.500 &      0.667 &      0.667 &      0.667 &      0.667 &      0.690 &      0.690 &      0.693 &      0.693 &      0.696 &      0.696            &      0.700 &      0.700 &      0.704 \\

1100 &      0.000 &      0.333 &      0.333 &      0.389 &      0.389 &      0.458 &      0.458 &      0.463 &      0.463 &      0.471 &      0.479 &      0.479            &      0.515 &      0.515 &      0.520 \\

1010 &      0.000 &      0.500 &      0.500 &      0.583 &      0.583 &      0.587 &      0.587 &      0.591 &      0.591 &      0.596 &      0.596 &      0.601            &      0.601 &      0.606 &      0.606 \\

1001 &      0.000 &      0.667 &      0.667 &      0.667 &      0.667 &      0.690 &      0.690 &      0.693 &      0.693 &      0.696 &      0.696 &      0.700            &      0.700 &      0.704 &      0.704 \\

0110 &      0.000 &      0.333 &      0.333 &      0.333 &      0.375 &      0.458 &      0.458 &      0.463 &      0.463 &      0.471 &      0.479 &      0.479            &      0.515 &      0.515 &      0.520 \\

0101 &      0.000 &      0.292 &      0.375 &      0.458 &      0.458 &      0.463 &      0.463 &      0.471 &      0.479 &      0.479 &      0.515 &      0.515            &      0.520 &      0.520 &      0.526 \\

0011 &      0.000 &      0.250 &      0.250 &      0.500 &      0.500 &      0.500 &      0.502 &      0.531 &      0.539 &      0.553 &      0.553 &      0.557            &      0.557 &      0.562 &      0.562 \\

1110 &      0.000 &      0.000 &      0.333 &      0.333 &      0.333 &      0.375 &      0.458 &      0.458 &      0.463 &      0.463 &      0.471 &      0.479            &      0.479 &      0.515 &      0.515 \\

1101 &      0.000 &      0.000 &      0.292 &      0.375 &      0.458 &      0.458 &      0.463 &      0.463 &      0.471 &      0.479 &      0.479 &      0.515            &      0.515 &      0.520 &      0.520 \\

1011 &      0.000 &      0.000 &      0.333 &      0.333 &      0.389 &      0.417 &      0.458 &      0.458 &      0.475 &      0.475 &      0.481 &      0.481            &      0.487 &      0.515 &      0.515 \\

0111 &      0.000 &      0.000 &      0.148 &      0.198 &      0.333 &      0.375 &      0.458 &      0.458 &      0.463 &      0.463 &      0.471 &      0.479            &      0.479 &      0.515 &      0.515 \\

1111 &      0.000 &      0.000 &      0.000 &      0.333 &      0.375 &      0.458 &      0.458 &      0.463 &      0.463 &      0.471 &      0.479 &      0.479            &      0.515 &      0.515 &      0.520 \\

\hline
MILP (s) &       0.04 &       0.08 &       0.09 &       0.14 &       0.25 &       0.42 &       0.60 &       0.96 &       1.96 &       3.83 &       8.01 &      14.96            &      35.72 &      77.47 &     165.86 \\

BBNode &       0.00 &       0.33 &       0.87 &       4.40 &      43.20 &     168.47 &     448.73 &    1002.93 &    2380.13 &    4682.67 &    9510.47 &   22676.53            &   48543.07 &   76276.20 &  125473.00 \\

Enum (s) &       0.00 &       0.00 &       0.06 &       0.94 &      16.46 &     273.02 &      $4.1 \cdot 10^3$ &   $6.1\cdot10^4$ &  
$ 9.2 \cdot10^5$  &   $1.4\cdot10^7$ & $ 2.1\cdot10^8$  &   $3.1\cdot10^9$ &   $4.7\cdot10^{10}$ & $ 7.0\cdot10^{11}$  &   $1.1\cdot10^{13}$    \\

\end{tabular}  
\end{table}
\end{landscape}

As expected, we clearly observe that the MILP approach is significantly faster than  complete enumeration, especially for larger values of the treatment length~$N$. From an application point of view, this is quite important since complete enumeration only allows for smaller values of~$N$ such as 6 in practice (see e.g. \citet{mira2015rational}) whereas the maximum probabilities might be obtained for longer treatment plans. This is more evident under EPM in which 14 of 15 initial states have higher probability of returning to the wild type in 10 steps compared to 6 steps (the only exception is  genotype 1000, which already has a deterministic path of going to the wild type). We note that even a small increase in maximum probabilities is crucial due to the critical nature of the application. 

\subsubsection{A Synthetic Dataset}
\label{sec:antibioticCompSynthetic}

In this section, we randomly generate growth data to construct the transition probability matrices in order to test the scalability of the proposed approach.
Based on our observation from the real dataset \citep{mira2015rational}, we have come up with a simple growth data generation procedure as follows:
\begin{equation*}\label{eq:randomGrowthGen}
\omega_{k,j} = \begin{cases}
0 & \text{w.p. } 1/3 \\
1 & \text{w.p. } 1/6 \\
2 & \text{w.p. } 1/2 
\end{cases}.
\end{equation*}
The intuition behind the parameters of this trinomial distribution is  that most  of the antibiotics are  effective to prevent the growth of a limited number of genotypes (one-third) whereas the growth of the majority of the genotypes are either unaffected (one-half) or slightly affected (one-sixth). 

%

Once we have the growth data, we construct the probability transition matrices under EPM as described in 
Section~\ref{sec:antibioticCompReal} and solve the MILP~\eqref{eq:MILP} for each initial state.  The average computational times are reported in Table~\ref{tab:RandomEPM}. Considering the size of the largest instance with $d=32$ states and $K=30$ drugs, and the fact that we only use a personal computer, an average computational time of about 7 minutes seems quite promising to demonstrate the scalability of the approach.  

\begin{table}[H]\small
\caption{Average run times in seconds for the randomly generated instances under  EPM.}\label{tab:RandomEPM}
\centering
\begin{tabular}{ccc|rrrrrr}

         $g$ &      $  d$ &        $  K$ &       $ N=5$ &     $   N=6 $&      $  N=7$ &      $  N=8$ &     $   N=9$ &     $  N=10$ \\
\hline

       4    &         16 &         15 &          0.23 &       0.45 &       0.64 &       1.09 &       2.17 &       5.11 \\

         4  &         16 &         20 &          0.46 &       0.67 &       0.96 &       1.85 &       3.92 &       8.67 \\

           4&         16 &         25 &         0.55 &       0.82 &       1.06 &       2.11 &       4.58 &       7.96 \\

         4  &         16 &         30 &        0.64 &       0.95 &       1.33 &       3.04 &       7.53 &      19.45 \\
\hline
         5 &         32 &         15 &         0.48 &       0.80 &       1.54 &       3.29 &       9.78 &      30.72  \\

         5  &         32 &         20 &         0.64 &       1.18 &       2.65 &       7.45 &      22.71 &      94.34 \\

          5 &         32 &         25 &       0.79 &       1.60 &       3.74 &      11.20 &      38.34 &     206.92 \\

          5 &         32 &         30 &     1.05 &       2.20 &       5.31 &      17.15 &      74.55 &     412.64   \\

\end{tabular}  
\end{table}

\section{Conclusions}
\label{sec:conc}

In this paper, we consider a class of optimization problems involving the multiplication of variable matrices to be selected from a family, and analyze such optimization problems depending on the structure of the matrix family. We  focus on the study of two interesting real-life applications:  the multi-layer thin films problem from material science and the antibiotics time machine problem from biology. We obtain compact-size mixed-integer quadratically constrained quadratic programming and mixed-integer linear programming   formulations for these two problems, respectively. Finally, we carried out an extensive computational study comparing the accuracy and efficiency of our proposed approach against heuristics and exhaustive search, which are quite common in the  literature.

We have  future research directions in both material science and biology applications. In this paper, we only focused on optimizing the reflectance of multi-layer thin films at a given wavelength. However, in many practical applications, a design which works well for a \textit{spectrum} of wavelengths is desired. Such an optimization problem can be modeled with an objective function  involving an integral and  infinitely many constraints. Although a finite-size approximate reformulation of this optimization model can be obtained, the resulting model seems to be quite challenging to solve and it likely requires a specialized solution algorithm.

For the biology application, a promising research direction seems to incorporate the uncertainty related to the growth rate measurements. In a recent study \citep{mira2017statistical}, the growth rates of different genotypes are measured 12 times under 23 different  antibiotics and dosages. Although these measurements lead to relatively small confidence intervals for the growth rates, the optimal treatment sequences obtained from each different measurement are quite different from each other. This observation motivates us to use robust optimization or (risk-averse) stochastic programming techniques to incorporate the uncertainty in the growth rate measurements as  a future research direction.

\section*{Acknowledgments}
The author wishes to thank Ali Rana At{\i}lgan for introducing him the problems in this paper and their fruitful discussions. The author also acknowledges Muhammed Ali Ke\c{c}eba\c{s} and K{\"u}r\c{s}at \c{S}endur's help in providing  the input data of the multi-layer thin films problem. 

\bibliography{references}

\end{document}